\documentclass[12pt]{article}
\usepackage[utf8]{inputenc}
\usepackage[T1]{fontenc}

\usepackage{needspace}
\usepackage{geometry}
\usepackage{amsthm, amsmath, amssymb}
\usepackage{thmtools}
\usepackage{thm-restate}

\usepackage{float}
\usepackage{xcolor}
\usepackage{enumerate}
\usepackage{hyperref}
\usepackage{graphicx}

\usepackage{subcaption}
\usepackage{ifthen}

\usepackage[color=red!30]{todonotes}

\hypersetup{
  pdftitle = {Sparse graphs without long induced paths},
  pdfauthor = {O. Defrain and J.-F. Raymond},
  colorlinks = true,
  linkcolor = black!30!red,
  citecolor = black!30!green
}

\newtheorem{theorem}{Theorem}[section]

\newtheorem{lemma}[theorem]{Lemma}
\newtheorem*{lemma*}{Lemma}

\newtheorem{conjecture}[theorem]{Conjecture}
\newtheorem{question}[theorem]{Question}

\newtheorem{remark}[theorem]{Remark}

\newtheorem{claim}[theorem]{Claim}

\theoremstyle{remark}

\newcommand{\IS}{\mathcal{N}}
\newcommand{\N}{\mathbb{N}}
\newcommand{\R}{\mathbb{R}}
\DeclareMathOperator{\pred}{{\sf pred}}
\DeclareMathOperator{\bpred}{{\sf bpred}}
\DeclareMathOperator{\tpred}{{\sf tpred}}
\DeclareMathOperator{\depth}{{\sf depth}}

\DeclareMathOperator{\ribs}{{\sf ribs}}

\DeclareMathOperator{\col}{{\sf col}}

\def\cqedsymbol{\ifmmode$\lrcorner$\else{\unskip\nobreak\hfil
\penalty50\hskip1em\null\nobreak\hfil$\lrcorner$
\parfillskip=0pt\finalhyphendemerits=0\endgraf}\fi} 
\newcommand{\cqed}{\renewcommand{\qed}{\cqedsymbol}}

\newcommand{\intv}[2]{\left \{ #1, \dots, #2 \right \}}

\renewenvironment{abstract}
{\small\vspace{-1em}
\begin{center}
\bfseries\abstractname\vspace{-.5em}\vspace{0pt}
\end{center}
\list{}{
\setlength{\leftmargin}{0.6in}%
\setlength{\rightmargin}{\leftmargin}}%
\item\relax}
{\endlist}

\title{%
Sparse graphs without long induced paths\thanks{This work was supported by the ANR projects GRALMECO (ANR-21-CE48-0004) and DISTANCIA (ANR-17-CE40-0015).}}
\author{Oscar Defrain\thanks{Aix Marseille Université, CNRS, LIS, Marseille, France}\and Jean-Florent Raymond\thanks{ CNRS, LIMOS, Université Clermont-Auvergne, Clermont-Ferrand, France.}}
\date{April 19, 2023}

\begin{document}

\maketitle

\begin{abstract}
    Graphs of bounded degeneracy are known to contain induced paths of order $\Omega(\log \log n)$ when they contain a path of order~$n$, as proved by Ne\v{s}et\v{r}il and Ossona de Mendez (2012). In 2016 Esperet, Lemoine, and Maffray conjectured that this bound could be improved to $\Omega((\log n)^c)$ for some constant $c>0$ depending on the degeneracy.
    
    We disprove this conjecture by constructing, for arbitrarily large values of~$n$, a graph that is 2-degenerate, has a path of order $n$, and where all induced paths have order $O((\log \log n)^2)$. We also show that the graphs we construct have linearly bounded coloring numbers. 

    \vskip5pt\noindent{}{\bf Keywords:} typical induced subgraphs, degeneracy, coloring numbers.
\end{abstract}

\section{Introduction}

In this paper logarithms are binary. For an integer $k$, a graph $G$ is said to be \emph{$k$-degenerate} if every subgraph of $G$ has a vertex of degree at most $k$. The minimum $k$ such that $G$ is $k$-degenerate is called the \emph{degeneracy} of $G$.
Bounded degeneracy forms a very general notion of graph sparsity that includes classes of bounded treewidth, planar graphs, and more generally graph classes of bounded expansion. Because its definition sets constraints on all subgraphs, it is also used in numerous inductive proofs and greedy algorithms, for instance for graph coloring.

In the \emph{Sparsity} textbook~\cite{nevsetvril2012sparsity}, Ne\v{s}et\v{r}il and Ossona de Mendez proved that in graph classes of bounded degeneracy, paths and induced paths are tied in the following sense.

\begin{theorem}[\cite{nevsetvril2012sparsity}]\label{th:nodm}
Let $k \in \N$ and let $G$ be a $k$-degenerate graph. If $G$ has a path of order $n$, then $G$ has an induced path of order at least
\[
\frac{\log \log n}{\log (k+1)}.
\] 
\end{theorem}

As every induced path is a path, the above result establishes a duality between long paths and induced paths for graphs of bounded degeneracy. Relating these objects is a very natural goal as they are among the most basic structures in graph theory.  So it should not be a surprise that it was already investigated decades earlier, when Galvin, Rival, and Sands proved the following general result.

\begin{theorem}[\cite{galvin1982ramsey}]\label{th:grs}
There is a function $f \colon \N^2 \to \R$ such that for every $s$, if $G$ is a $K_{s,s}$-subgraph free graph that has a path of order $n$, then $G$ has an induced path of order at least $f(n,s)$.
\end{theorem}

This result can be related to the study of \emph{unavoidable} (or \emph{typical}) substructures of a graph class $\mathcal{G}$ (i.e., substructures that can be found in every member of $\mathcal{G}$), that originates from Ramsey's theorem and is currently receiving increasing attention when these structures are induced subgraphs.

If we consider graph classes that are closed under taking induced subgraphs, the above result describes the most general setting where such a function $f$ could exist. Indeed, cliques and bicliques have paths that visit all their vertices, yet they do not have induced paths on more than 2 or 3 vertices, respectively. So any hereditary class where a result such as Theorem~\ref{th:grs} holds should exclude arbitrarily large cliques and bicliques, and this amounts to excluding arbitrarily large biclique subgraphs.
On the other hand, the function~$f$ defined in the proof of Theorem~\ref{th:grs} is not optimal for many $K_{s,s}$-subgraph free graph classes (see below) and possibly not optimal neither for the statement of the theorem.

Relating the maximum orders of paths and induced paths is also of interest in the context of graph sparsity theory because of the following connection to the parameter treedepth.
For a graph $G$, if we denote by $\ell$ the maximum order of an induced path in it and by $t$ its treedepth then
\[
\log (\ell+1) \leq t \leq \ell.
\]
Therefore in graph classes where bounds such as that of Theorem~\ref{th:grs} hold, treedepth is not only tied to long induced paths but also to long paths.
We refer to Chapter~6 in \cite{nevsetvril2012sparsity} for a definition of treedepth, a proof of the above inequality, and a discussion on the links between induced paths and treedepth.

Following the proof of Theorem~\ref{th:nodm}, the authors of \emph{Sparsity} ask about the best possible bound.
\begin{question}[{\cite[Problem 6.1]{nevsetvril2012sparsity}}]\label{que:nodm}
For every $k\in \N$, what is the maximum function $f_k\colon \N\to \N$ such that every $k$-degenerate graph $G$ that has a path of order $n$ has an induced path of order at least $f_k(n)$?
\end{question}

To the best of our knowledge, the current bounds on $f_k$, for $k\geq 3$ and $n \geq 1$,\footnote{The left-hand bound also holds for $k=2$. When $k<2$ we trivially have $f_k(n) = n$.} are the following:
\[
\frac{\log \log n}{\log (k+1)} \quad \leq \quad  f_k(n) \quad \leq \quad k(\log n)^{2/k}.
\]
The left-hand inequality is Theorem~\ref{th:nodm} and the right-hand bound, due to Esperet, Lemoine, and Maffray, is proved by the construction, for every $t\geq 4$, of an infinite sequence of chordal graphs of clique number $t$ (hence $(t-1)$-degenerate) where the induced paths can be shown to be all short~\cite{esperet2017long}.\footnote{We note that the analysis of the construction of \cite{esperet2017long} cannot be improved much as such graphs are known to have induced paths of order $\smash{\Omega((\log n)^{1/t})}$, as later proved in~\cite{hilaire2022long}.} The gap between the upper- and lower-bounds above is exponentially wide so it is indeed an interesting goal to investigate the order of magnitude of $f_k$.

Actually, for several classes of degenerate graphs the lower-bound can be drastically improved. 
A representative example is the following result of Esperet et al.~stating that planar graphs contain induced paths of polylogarithmic order.

\begin{theorem}[\cite{esperet2017long}]
\label{th:elm}
    There is a constant $c$ such that if $G$ is a planar graph and has a path of order $n$, then $G$ has an induced path of order at least $c \sqrt{\log n}$.
\end{theorem}

The result holds more generally for graph classes of bounded Euler genus, as shown in the same paper. This led its authors to conjecture, as an answer to Question~\ref{que:nodm}, that polylogarithmic bounds could hold for graphs of bounded degeneracy.

\begin{conjecture}[{\cite[Conjecture 1.1]{esperet2017long}}]\label{conj:esp}
For every $k\in \N$ there is a constant $c>0$ such that if $G$ is a $k$-degenerate graph and has a path of order $n\geq 1$, then $G$ has an induced path of order at least $(\log n)^{c}$.
\end{conjecture}

Recently, Hilaire and the second author extended the catalog of graph classes known to satisfy Conjecture~\ref{conj:esp} with the following result.

\begin{theorem}[\cite{hilaire2022long}]
\label{th:claire}
    For every graph class $\mathcal{G}$ that is closed under topological minors there is a constant $c>0$ such that if a graph $G \in \mathcal{G}$ has a path of order $n\geq 1$, it has an induced path of order at least $(\log n)^c$.
\end{theorem}

In this paper we disprove Conjecture~\ref{conj:esp} by proving the following statement.

\begin{restatable}{theorem}{main}\label{th:main}
There is a constant $c$ such that for infinitely many integers $n$, there is a 2-degenerate graph $G$ that has a path of order $n$ and no induced path of order more than $c (\log \log n)^2$.
\end{restatable}

Note that for every $k\geq 2$, every 2-degenerate graph is also $k$-degenerate. Hence Theorem~\ref{th:main} shows that $f_k(n) = O((\log \log n)^2)$ and indeed disproves Conjecture~\ref{conj:esp} for every value of~$k$.

A natural weakening of Conjecture~\ref{conj:esp} could be to consider graph classes of bounded coloring numbers. For every $r\in \N_{\geq 1}$, the $r$-coloring number is a graph invariant that can be seen as a distance-$r$ version of the degeneracy.\footnote{In this paper we only deal with the coloring numbers sometimes known as strong coloring numbers. See Section~\ref{sec:colnum} for a formal definition.}
We say that a graph class has \emph{bounded coloring numbers} if
there is a function $d\colon \N\to\N$ such that for every $r\in \N_{\geq 1}$, the $r$-coloring number of any $G\in \mathcal{G}$ is at most $d(r)$.
This notion plays an important role in the theory of graph sparsity due to its strong ties to the concept of bounded expansion.\footnote{We will not use the concept of bounded expansion later in this paper so we refrain to define it and refer to~\cite{nevsetvril2012sparsity} for a formal definition and more about its links with the coloring numbers.} In particular, as observed by Norin, if a graph class has coloring numbers bounded by a linear function, then this class has expansion bounded by a linear function (see~\cite[Observation 10]{Esperet2018polyexp}).

As the 1-coloring number is exactly the degeneracy, bounding the coloring numbers of a graph class is substantially stronger than merely bounding the degeneracy. The following shows that even with the requirement that the considered graphs come from a class of linearly bounded coloring numbers, Conjecture~\ref{conj:esp} does not hold.
Let $\mathcal{G}$ denote the class of the graphs from Theorem~\ref{th:main}.

\begin{restatable}{theorem}{linecol}\label{th:linecol}
$\mathcal{G}$ has coloring numbers bounded from above by a linear function.
\end{restatable}

\paragraph{Organization of the paper.} In the next section we introduce the basic terminology used in this paper. The construction is described in Section~\ref{sec:blozup}. Section~\ref{sec:proprigel} is devoted to the proofs of general properties on the graphs we construct, while in Section~\ref{sec:riiibs} we discuss the behavior of their induced paths. We finally prove Theorem~\ref{th:main} in Section~\ref{sec:smallip}. 
In Section~\ref{sec:colnum} we prove Theorem~\ref{th:linecol}. Directions for future research are given in Section~\ref{sec:ohpeine}.

\section{Preliminaries}\label{sec:prel}

Unless stated otherwise we use standard graph theory terminology.

\paragraph{Trees.}
Our construction is defined starting from trees. 
To avoid ambiguity we use the synonym \emph{node} when referring to the vertices of a tree.
A \emph{rooted tree} is a tree with a distinguished node called its \emph{root}.
The \emph{leaves} are the nodes with degree one and different from the root. The other vertices, that are neither the root nor leaves, are called \emph{internal nodes}.
For every pair $s, t$ of nodes of a rooted tree we write $s \preceq t$ if $s$ lies on the unique path connecting $t$ to the root, and $s \prec t$ if in addition $s\neq t$.
We say that $t$ is a \emph{descendant} of $s$ and that $s$ is an \emph{ancestor} of $t$ whenever $s \preceq t$; $t$ is a \emph{child} of $s$ and $s$ is the \emph{parent} of $t$ if in addition $t$ and $s$ are neighbors.

The \emph{depth} of a node $s$ in a rooted tree is the order of the unique path from $s$ to the root, denoted $\depth(s)$; note in particular that the root has depth $1$.

For every $p \in \N$, the \emph{complete binary tree} $T$ of depth $p$ is the rooted tree defined as follows:
\begin{itemize}
    \item if $p=1$ then $T=K_1$, rooted at its unique vertex;
    \item otherwise $p>1$ and $T$ can be obtained from two disjoint copies of the complete binary tree of depth $p-1$ by adding a new vertex $v$ adjacent to their roots and rooting $T$ at $v$.
\end{itemize}
 
\begin{remark}\label{rem:bintree}
    The complete binary tree of depth $p$ has order $2^p -1$.
\end{remark}

\section{The construction}\label{sec:blozup}

We shall describe, for arbitrarily large $n$, the construction of a 2-degenerate graph with a path of order $n$---which will in fact be Hamiltonian---but no induced path of order $c(\log \log n)^2$ for some constant~$c$.

\paragraph{Tree blow-up.}
The \emph{blow-up} $H$ of a complete binary tree $T$ is the graph obtained from $T$ as follows (see also Figure~\ref{fig:bu}):
\begin{itemize}
    \item for every node $s$ of $T$ we create a clique $K^s$ on $3$ vertices and choose an injection $\mu_s \colon N_{T}(s) \to E(K^s)$; 
    \item for every edge $st$ in $T$ with $s\preceq t$ we add to $H$ a first path $L(s,t)=u_1x_1y_1v_1$ on four vertices connecting an endpoint $u_1$ of $\mu_s(t)$ to an endpoint $v_1$ of $\mu_t(s)$, and a second path $R(s,t)=u_2x_2y_2v_2$ on four vertices connecting the other endpoint $u_2$ of $\mu_s(t)$ to the other endpoint $v_2$ of $\mu_t(s)$; we stress the fact that the endpoints of $\mu_s(t)$ and $\mu_t(s)$ are chosen arbitrarily, and that vertices $x_1,x_2,y_1,y_2$ are added during the construction.
\end{itemize}

\begin{figure}
    \centering
    \includegraphics[scale=1.1]{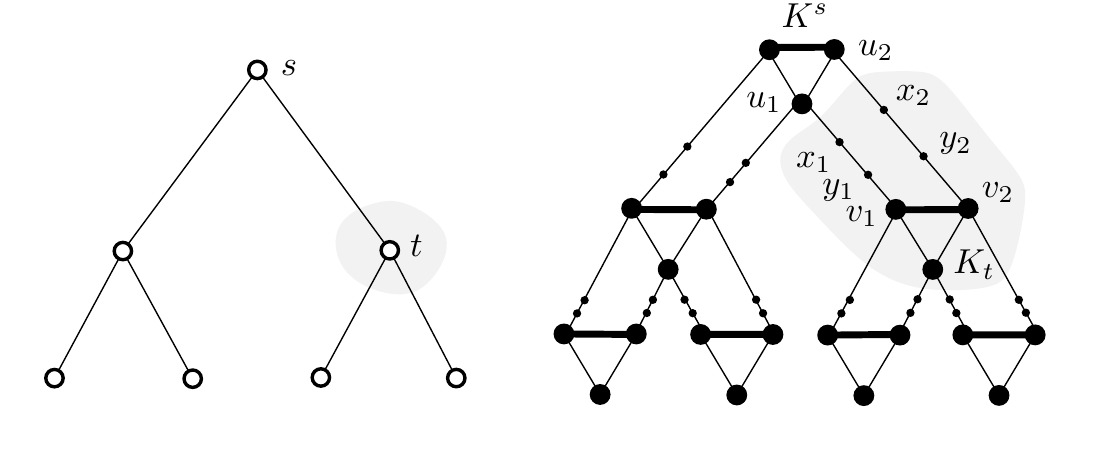}
    \caption{A complete binary tree of height 3 (left) and its blow-up (right) where large vertices represent cliques vertices, small ones represent subdivision vertices, and the bold edges represent top edges. In the gray area, the bag of $t$, i.e., the set of vertices that originated from $t$ in the construction of the blow-up.}
    \label{fig:bu}
\end{figure}

In the definition above we call \emph{top-predecessors} of the node $t$ the vertices in the set $\tpred(t)=\{x_1,x_2\}$, \emph{bottom-predecessors} of $t$ the vertices in  $\bpred(t)=\{y_1,y_2\}$, and shall note $\pred(t)$ the union of these two sets.
We call \emph{bag} of node $t$ the set $V(K^t)\cup \pred(t)$.
The edges (respectively vertices) of $H$ that belong to $K^s$ for some node $s \in V(T)$ are called \emph{clique edges} (respectively \emph{clique vertices}), while the others are called \emph{tree edges} of $H$ (respectively \emph{subdivision vertices}).\footnote{Subdivision vertices are used later in our construction in order to achieve low degeneracy. It seems possible to avoid them if we only aim for a 4-degenerate graph.}
For a node $t \in V(T)$ with parent $s$, the \emph{top edge} of the clique $K^t$ in $H$ is the edge $\mu_t(s)$.
The \emph{top edge} of the root $r$ of $T$ is the only edge of $K^r$ that has no preimage by $\mu_r$, and it shall also be referred to as the \emph{root edge} of $H$.

\paragraph{Nested intervals systems.}
Let $h\colon \N_{\geq 1} \to \N$ be the function defined for every $\ell \in \N_{\geq 1}$ by the following recurrence relation:
\begin{equation}\label{eq:rec}
\left \{%
\begin{array}{l}
    h(1) = 3\\
    h(\ell) = 2 + 2 \cdot  h(\ell-1)\quad \text{for every}\ \ell>1.
\end{array}
\right .
\end{equation}
Observe that $h(\ell)$ may as well be computed by the following explicit formula for every $\ell \in \N_{\geq 1}$:
\begin{equation}\label{eq:exp}
h(\ell)=5\cdot 2^{\ell-1} - 2.
\end{equation}

If $X$ is a set of pairs of integers and $i \in \N$, we denote by $X\!\oplus i$ the set $\{(x+i, x'+i) : (x,x') \in X\}$.
For every $\ell \in \N_{\geq 1}$ the set $\IS_\ell$ is recursively defined as follows:
\[
\left \{
\begin{array}{ll}
     \IS_1 = \{(1,3)\},\ \text{and}&  \\
     \IS_\ell = \{(1, h(\ell))\}\ \cup\ (\IS_{\ell-1}\oplus 1)\ \cup\ (\IS_{\ell-1} \oplus (h(\ell-1) +1))& \text{if}\ \ell>1.
\end{array}
\right.
\]
The elements of $\IS_\ell$ are called \emph{intervals}.
Intuitively $\IS_\ell$ is obtained by taking two copies of $\IS_{\ell-1}$ (appropriately shifted so that they start after integer 1 and do not intersect) and adding a new interval $(1,h(\ell))$ containing the two copies. See Figure~\ref{fig:intervals} for an illustration.

\begin{figure}[H]
    \centering
    \includegraphics[scale=1.1]{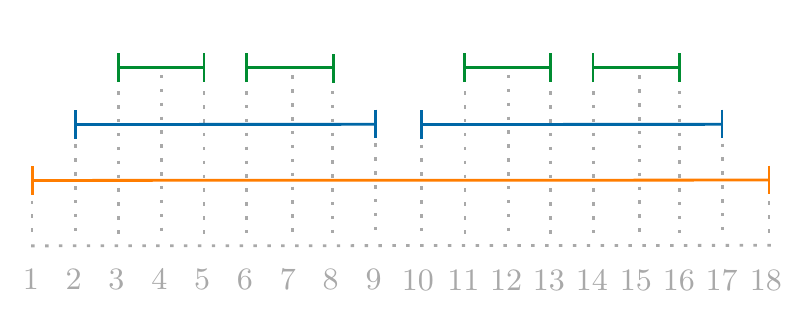}
    \caption{The intervals of $\IS_3$ with intervals of rank $1,\ 2$, and $3$ depicted from top to bottom in green, blue, and orange, respectively.}\label{fig:intervals}
\end{figure}
The following easy properties of $\IS_\ell$ can be proved by a straightforward induction:

\begin{remark}\label{rem:nestint}
For every $\ell\in \N_{\geq 1}$ the following holds:
\begin{enumerate}
    \item the endpoints of the intervals in $\IS_\ell$ range from $1$ to $h(\ell)$;
    \item the intervals of $\IS_\ell$ all have distinct endpoints;
    \item every interval of $\IS_\ell$ is of the form $(i, i+h(a)-1)$ for some $i \in \intv{1}{h(\ell)}$ and $a \in \intv{1}{\ell}$;
    \item for every interval $(i,j)$ in $\IS_\ell$ there is no other interval $(i', j')$ in $\IS_\ell$ such that $i<i'<j<j'$ (informally, intervals do not cross).
\end{enumerate}
\end{remark}

We call \emph{rank} of an interval $(i,j)\in \IS_\ell$ the aforementioned integer $a \in \intv{1}{\ell}$ such that $j = i+h(a)-1$.

\begin{figure}[H]
    \centering
    \includegraphics[scale=1.1]{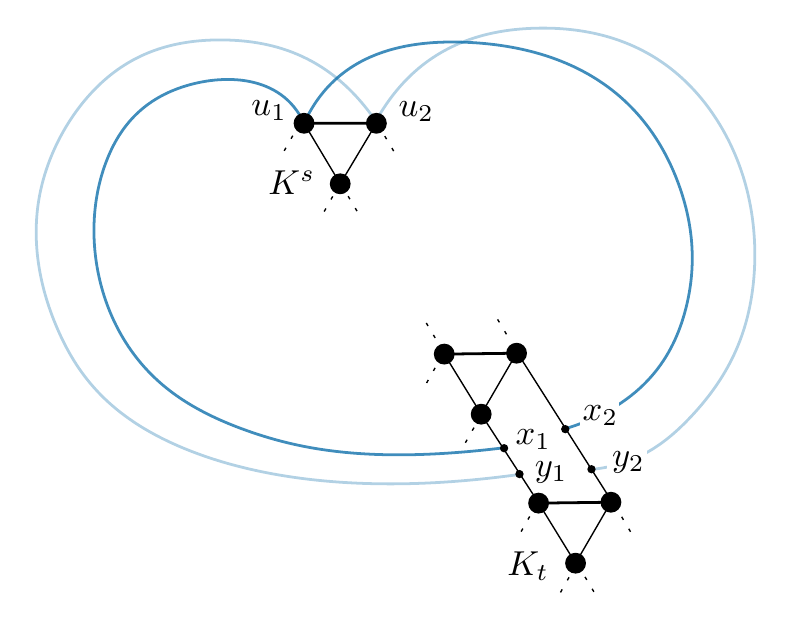}
    \caption{The edges in the set $\ribs_\ell(s,t)$ for two nodes $s,t\in V(B_\ell)$ such that $s \prec t$.}\label{fig:ribs}
\end{figure}

\paragraph{Ribs.}
For every $\ell \in \N_{\geq 1}$ we denote by $B_\ell$ the complete binary tree of depth $h(\ell)$ and by $H_\ell$ its blow-up.
The next construction is illustrated in Figure~\ref{fig:ribs}.
Let $\ribs_\ell$ be the function defined on every pair $(s,t) \in V(B_\ell)^2$ of nodes such that $s \prec t$ as the following set of edges over $V(H_\ell)$:\footnote{Note that these edges do not exist in $H_\ell$. We define this set in order to later construct a graph by adding these edges to $H_\ell$.}
\[
  \ribs_\ell(s,t) = 
  \{u_1x_1,
  u_1x_2,
  u_2y_1,
  u_2y_2\},
\]%
where $\{x_1, x_2\} = \tpred(t)$, $\{y_1, y_2\} = \bpred(t)$, and $u_1u_2$ is the top edge of $K^s$.
Note that such a function is not unique,\footnote{It depends on which vertex in $\tpred(t)$ we decided to call $x_1$ in the construction of the blow-up, and similarly for $\bpred(t)$ and the top edge of $s$.} however any will suffice for our purpose; see Figure~\ref{fig:ribs} for an example.
The graph $G_\ell$ is obtained from $H_\ell$ after the addition of the set of edges $\ribs_\ell(s,t)$ for every pair of nodes $s,t\in V(B_\ell)$ of respective depth $i,j$ such that $s\prec t$ and $(i,j)\in \IS_\ell$.
We call an edge $uv$ in that set a \emph{rib}.
Hence the edges of $G_\ell$ are partitioned into tree edges, clique edges, and ribs; see Figure~\ref{fig:ribbed} for a representation of $G_\ell$.

\begin{figure}[H]
    \centering
    \includegraphics[scale=1.1]{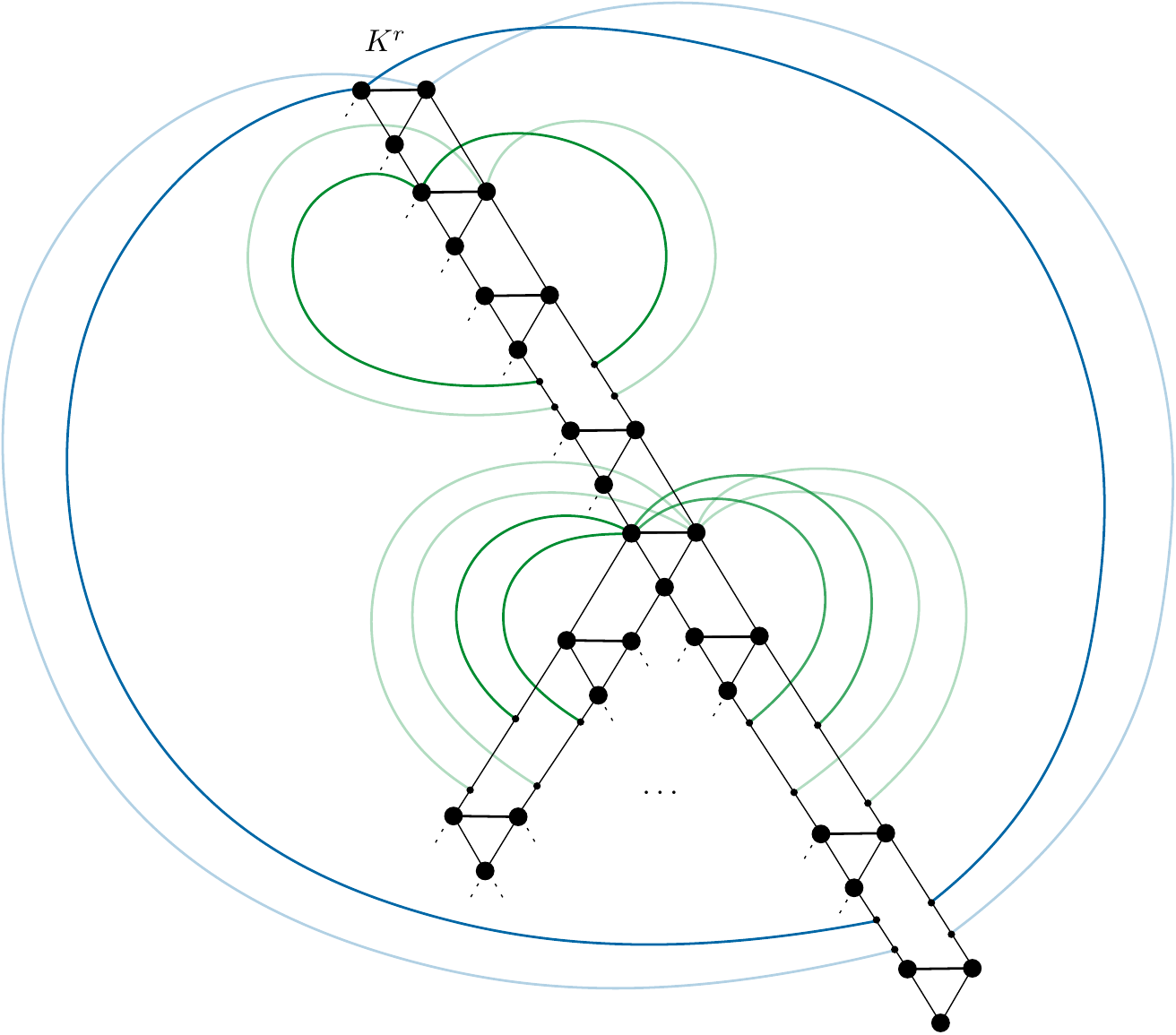}
    \caption{A partial representation of the graph $G_2$ with root clique $K^r$. 
    For readability, subdivision vertices that are not incident to ribs are omitted. Ribs corresponding to intervals of rank $1$ and $2$ are depicted in green and blue, respectively. 
    }\label{fig:ribbed}
\end{figure}

Because much of the structure of $G_\ell$ is inherited from $B_\ell$, some properties of vertices of $G_\ell$ will be conveniently defined via this tree. For this purpose we define $\pi \colon V(G_{\ell}) \to V(B_{\ell})$ as the function mapping every clique vertex $u$ of $G_{\ell}$ to the corresponding node $s$ in $B_{\ell}$ (i.e., the unique node $s$ such that $u \in V(K^s)$) and every subdivision vertex $v$ of $G_{\ell}$ to the unique node $t$ in $B_{\ell}$ such that $v \in \pred(t)$. We define the depth $\depth(u)$ of a vertex $u\in V(G_\ell)$ as the depth of its corresponding node $\pi(u)$ in $B_\ell$.
The following stems from the definition of a blow-up.
\begin{remark}\label{rem:preimage-partition}
    The set $\{\pi^{-1}(t) : t\in V(B_\ell)\}$ defines a partition of $V(G_\ell)$ with $|\pi^{-1}(t)|=3$ when $t$ is the root of $B_\ell$, and $|\pi^{-1}(t)|=7$ otherwise.
\end{remark}

Note that $\pi^{-1}(t)$ is exactly denotes the bag of $t$; see Figure~\ref{fig:bu} for an example.

\section{The properties of \texorpdfstring{$G_\ell$}{Gl}}
\label{sec:proprigel}

A \emph{Hamiltonian path} in a graph is a path that visits all its vertices.

\begin{lemma}\label{lem:ham}
    For every integer $\ell \geq 1$, the graph $G_\ell$ has a Hamiltonian path.
\end{lemma}

\begin{proof}
    In this proof we denote by $F_\ell$ the blow-up of the complete binary tree of depth~$\ell$, for every $\ell \in \N_{\geq 1}$.
    We will show the following.

    \begin{quote}
        \it For every integer $\ell\geq 1$, $F_\ell$ has a Hamiltonian path linking the two endpoints of its root edge.
    \end{quote}

    Note that the above immediately implies the desired statement as $H_\ell = F_{h(\ell)}$ is a spanning subgraph of $G_\ell$.
    The proof is by induction on $\ell$. 
    In the case $\ell=1$, $F_\ell$ is the triangle and the statement clearly holds.

    Let us assume from now on that $\ell>1$ and that the property holds for $\ell-1$.
    It follows from the definition that $F_\ell$ can be constructed from a triangle $z_1z_2z_3$ (where $z_1z_2$ will be the root) and two disjoint copies $F,F'$ of $F_{\ell-1}$ of respective root edges $xy$ and $x'y'$ by adding paths $P(u,v)$  joining $u$ to $v$ for every $(u,v) \in \{(z_1, x), (y, z_3), (z_3, x'), (y', z_2)\}$. Then a Hamiltonian path $P$ of $F_\ell$ from $z_1$ to $z_2$ can be constructed as follows: from $z_1$ we follow $P(z_1, x)$, then the Hamiltonian path of $F$ between $x$ and $y$ (given by the induction hypothesis) leads us to $y$ and we reach $z_3$ via $P(y, z_3)$. The second half of the path through $F'$ up to $z_2$ is symmetric. This proves the above statement, and hence the lemma.
\end{proof}

\begin{lemma}\label{lem:taillede}
For every integer $\ell\geq 1$,    $|V(G_\ell)| \geq 2^{2^{\ell+1}}$.
\end{lemma}

\begin{proof}
    Let $r$ denote the root of $B_\ell$. As observed in Remark~\ref{rem:preimage-partition}, 
    the vertex set of $G_\ell$ can be partitioned into sets $\pi^{-1}(s)$ of size either 3 if $s$ is the root, or 7 otherwise. Hence
    \begin{align*}
    |V(G_\ell)| &= 7\cdot|B_\ell| - 4\\
    &=7\cdot 2^{5\cdot 2^{\ell-1}-2}-11 &\text{(using Remark\ \ref{rem:bintree} and~\eqref{eq:exp})}\\
&\geq 2^{2^{\ell+1}}.
    \end{align*}
\end{proof}

\begin{lemma}\label{lem:2deg}
    For every integer $\ell\geq 1$,
    the graph $G_\ell$ is 2-degenerate.
\end{lemma}

\begin{proof}
    Let $G$ be a (non-empty) induced subgraph of $G_\ell$ and let $D$ denote the set of vertices of maximal depth $d$ in $G$.
    We define $4$ subsets of $D$ as follows:
    \begin{itemize}
        \item $D_1$ is the set of clique vertices in $D$ that are not incident with the top edge of their clique;
        \item $D_2$ is the set of clique vertices in $D$ that are incident with the top edge of their clique;
        \item $D_3$ is the set of bottom-predecessors in $D$;
        \item $D_4$ is the set of top-predecessors in $D$.
    \end{itemize}
    Clearly this forms a partition of $D$ and one of these sets is not empty. Let $i$ be the minimum index such that $D_i$ is not empty and let $v \in D_i$. Let us show that $\deg_G(v) = 2$. We distinguish cases according to the value of $i$:
    \begin{enumerate}[{Case }${i}=1$:]
        \item In $G_\ell$, 
        the vertex $v$ has $\leq 4$ neighbors: the two other elements of $K^{\pi(v)}$ and at most two vertices $x,y \in \tpred(s)\cup \tpred(t)$ for $s$ and $t$ the children of $\pi(v)$ in $B_\ell$, if any. As $\depth(s) = \depth(t) > d$ and by maximality of $d$, none of $x$ and $y$ belongs to $G$. Hence $v$ has at most two neighbors in this graph.
        \item Recall that if $uu'$ is a rib of $G_\ell$ and $u$ is a clique vertex, then $\depth(u) < \depth(u')$. Therefore, as $v$ is a clique vertex and $d$ is maximum, no rib of $G$ is adjacent to $v$. Besides, at most one other vertex of $K^{\pi(v)}$ belongs to $G$, as otherwise $D_1$ would not be empty. Finally, only one other vertex is adjacent to $v$ in $G_\ell$ (an element of $\bpred(\pi(v))$), so $\deg_{G}(v)\leq 2$.
        \item A vertex $u\in \bpred(s)$ for some $s \in V(B_\ell)$ has the following neighbors in $G_\ell$: a vertex of the top edge of $K^s$, a vertex of $\tpred(s)$ and possibly a third neighbor connected via a rib (see Remark~\ref{rem:nestint}). The first type of neighbor does not exist in $G$ as $D_2 = \emptyset$, so here again $\deg_G(v) \leq 2$.
        \item Similarly as above, a vertex $u\in \tpred(s)$ for some $s \in V(B_\ell)$ has the following neighbors in $G_\ell$: a vertex of $\bpred(s)$, a vertex of $K^t$ (where $t$ denotes the parent of $s$ in $B_\ell$) and possibly a neighbor connected via a rib. The first type of neighbor does not exist in $G$ as $D_3 = \emptyset$, and again $\deg_G(v) \leq 2$.
    \end{enumerate}
We proved that every subgraph of $G$ contains a vertex of degree at most 2. This shows that $G_\ell$ is indeed 2-degenerate.
\end{proof}

\section{Ribs, sources, and their properties}
\label{sec:riiibs}

In the rest of the paper we fix $\ell\in\mathbb{N}_{\geq 1}$.
A node $s$ of $B_\ell$ is a \emph{source} if there is an interval $(i,j) \in \IS_\ell$ such that $s$ has depth~$i$.
Intuitively this means that in $G_\ell$ the clique $K^s$ will send ribs to vertices deeper (i.e., in $\pred(t)$ for the descendants $t$ of $s$ of depth~$j$) in~$G_\ell$.
For a source $s$ of $B_\ell$ the \emph{rank} of $s$ is defined as the rank of $(i,j)$, i.e., the integer $a\in\intv{1}{\ell}$ such that $j - i + 1 = h(a)$. We denote by $B_\ell(s)$ the subtree of $B_\ell$ rooted at $s$ and of depth~$h(a)$. This means that the leaves of $B_\ell(s)$ are exactly those vertices $t$ such that in $G_\ell$ $K_s$ sends ribs to the predecessors of $t$.
The graph $G_\ell(s)$ is defined as the subgraph of $G_\ell$ induced by $\pi^{-1}(V(B_\ell(s)))$.

Let $Q$ be an induced path of $G_\ell$. The following definition is crucial in the rest of the proof:
a source $s$ is said to be \emph{$Q$-special} if $Q$ has two vertices $u,v$ such that $u$ is part of the top edge of $K^s$ and $\pi(v)$ is an internal node of $B_\ell(s)$.
For every vertex $v$ of $G_\ell$, we define $\tau(v)$ as the minimum rank of a source $s$ such that $\pi(v)$ is an internal node of $B_\ell(s)$. Notice that if $s$ is the root or a leaf of $B_\ell$ then $\tau$ is not defined: we set $\tau(v) = \ell+1$ in this case.

\begin{remark}\label{rem:intertau}
    Let $v\in V(G_\ell)$ and let $s$ be a source of $B_\ell$ of rank $a\in \intv{1}{\ell}$.
    \begin{enumerate}
        \item if $\pi(v)$ is an internal node of $B_\ell(s)$ then $\tau(v)\leq a$;
        \item if $\pi(v) = s$ or $\pi(v)$ is a leaf of $B_\ell(s)$, then $\tau(v) = a+1$.
    \end{enumerate}
\end{remark}

In a graph $G$, we say that a set $X\subseteq V(G)$ \emph{separates} two sets $Y,Z\subseteq V(G)$ if every path from a vertex of $Y$ to a vertex of $Z$ intersects $X$.

\begin{lemma}\label{lem:separation}
Let $s$ be a source of $B_\ell$, let $L$ be the set of leaves of $B_\ell(s)$. Then
\[
X_s = V(K^s) \cup \bigcup_{t\in L}\pi^{-1}(t)
\]
separates $V(G_\ell(s))\setminus \pred(s)$ from the other vertices of $G_\ell$.
\end{lemma}

\begin{proof}
The node $s$ is a source so there is an interval $(i,j)\in \IS_\ell$ such that $\depth(s) = i$ and $\depth(t) = j$ for every leaf $t$ of $B_\ell(s)$.
By Remark~\ref{rem:nestint} there is no other interval $(i',j') \in \IS_\ell$ such that $i\leq i'\leq j\leq j'$ or $i' \leq i\leq j'\leq j$ so in the construction of $G_\ell$ there is no rib with only one endpoint in $G_\ell(s)$. Therefore every edge leaving $V(G_\ell(s))\setminus \pred(s)$ (i.e., with only one endpoint in that set) is a tree edge.
Let $v \in \pi^{-1}(t)$ for some internal node of $B_\ell(s)$.
By construction every vertex that is connected to $v$ via a tree edge is either in the bag of $t$ or in the bag of a neighbor of $t$ (possibly $s$ or a leaf of $B_\ell(s)$).
In particular, $v$ has no neighbor in $(V(G_\ell) \setminus V(G_\ell(s)))\cup \pred(s)$. This shows that a path from $v$ to $(V(G_\ell)\setminus V(G_\ell(s))) \cup \pred(s)$ goes through the set $X_s$, as claimed.
\end{proof}

We remark the following as a consequence of the definition of depth. 

\begin{remark}\label{rem:depth-decreasing}
If $uv$ is an edge of $G_\ell$ such that $\depth(u)< \depth(v)$ then either $uv$ is a tree edge and $v$ a top-predecessor, or it is a rib of source $\pi(u)$.
\end{remark}

\begin{lemma}\label{lem:tau-decreasing}
If $uv$ is an edge of $G_\ell$ such that $\tau(u) > \tau(v)$ then $uv$ is a tree edge and $\tau(u)=\tau(v)+1$.
If in addition $\depth(v)\geq \depth(u)$, then $\pi(u)$ is a source and $u \in V(K^{\pi(u)})$.
\end{lemma}

\begin{proof}
Let $s \in V(B_\ell)$ be a source of minimum rank such that $\pi(v)$ is an internal node of $B_\ell(s)$. Such a node exists as $\tau(v)<\tau(u)\leq \ell+1$ so $\tau(v)\leq \ell$, hence $\pi(v)$ is in particular an internal node of $B_\ell(r)$ where $r$ is the root of $B_\ell$.
By definition the rank of $s$ is $\tau(v)$.

Observe that $\pi(u)$ cannot be an internal node of $B_\ell(s)$ as otherwise by Remark~\ref{rem:intertau} we would have $\tau(u) \leq \tau(v)$. As proved in Lemma~\ref{lem:separation}, $X_s$ (as defined in the statement of that lemma) separates the vertices $w$ such that $\pi(w)$ is an internal node of $B_\ell(s)$ from the rest of the graph, hence $u\in X_s$. By Remark~\ref{rem:intertau}, $\tau(u) = \tau(v)+1$ and we derive that $uv$ is a tree edge.

In the case where $\depth(u) \leq \depth(v)$, from $u\in X_s$ we deduce $\pi(u) = s$ (hence it is a source). Since only vertices of $K^s$ have neighbors in bags of internal nodes of $B_{\ell}(s)$, we furthermore have $u\in V(K^s)$.
\end{proof}

\begin{lemma}\label{lem:depth-root}
Let $Q$ be an induced path of $G_\ell$. 
Then there is a unique node $t\in V(B_\ell)$ of minimum depth subject to $\pi^{-1}(t) \cap V(Q)\neq \emptyset$.
\end{lemma}

\begin{proof}
Let us assume towards a contradiction that there are two such nodes $t,t'$. As they have the same depth, they are not comparable in $B_\ell$. Besides, observe that every edge of $G_\ell$ connects vertices whose image by $\pi$ is $\preceq$-comparable. Therefore the subpath of $Q$ linking the bag of $t$ to that of $t'$ contains a vertex of the bag of $t''$ for some common ancestor $t''$ of $t$ and $t'$, which contradicts the minimality of the depth of those vertices.
\end{proof}

\begin{lemma}\label{lem:tau-constant}
There is a constant $c_{\ref{lem:tau-constant}}$ such that if $Q=u_1\dots u_q$ is an induced path of $G_\ell$ with $\tau(u_i)=\tau(u_1)$ for all $2\leq i\leq q$, then $|Q|\leq c_{\ref{lem:tau-constant}}$.
\end{lemma}

\begin{proof}
Let $Q$ be such an induced path and $s$ be the source of rank $a=\tau(u_1)$ such that $\pi(u_1)$ is an internal node of $B_\ell(s)$.
By Remark~\ref{rem:intertau} we have $\tau(v) = a+1$ if $\pi(v) = s$ or $\pi(v)$ is a leaf of $B_\ell(s)$. By Lemma~\ref{lem:separation} this implies that $Q$ is contained in the union of the bags of the internal nodes of $B_\ell(s)$. Let us call this union $Z$. If $a=1$ there are at most two such nodes so clearly $Q$ has bounded size. So we may now assume $a>1$.

As $s$ is a source, there is an interval $(i,i'') \in \IS_\ell$ such that $\depth(s) = i$ and $\depth(t) = i''$ for every leaf $t$ of $B_\ell(s)$. We call $s_1$ and $s_2$ the two children of~$s$. Let $i' = i+ h(a-1)$. By construction $(i+1, i'), (i'+1, i''-1) \in \IS_\ell$; see Figure~\ref{fig:onigiri} for a representation of $B_\ell(s)$ by depths $i,i+1,\dots, i''$. (Notice that $i''\geq i+2$, by the third item of Remark~\ref{rem:nestint}, the definition of the function $h$ and the fact that $a>1$.)
Let $D$ be the set of descendants of $s$ that have depth $i'+1$ in $G_\ell$.

\begin{figure}[H]
    \centering
    \includegraphics[scale=1.1]{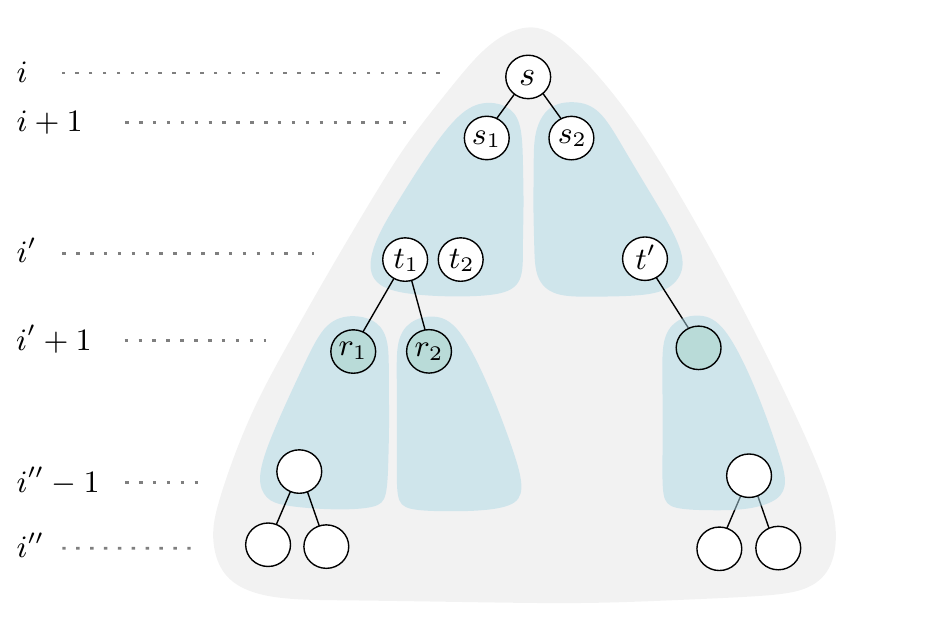}
    \caption{The situation of Lemma~\ref{lem:tau-constant}; full nodes represent elements in $D$.}\label{fig:onigiri}
\end{figure}

Let $r \in D$ and let $t$ be a leaf of $B_\ell(r)$. Then $t$ lies at depth $i''-1$ in $B_\ell$. Observe that in $G_\ell$ each edge incident to the bag of $t$ is of one of the following types:

\begin{itemize}
    \item a tree edge to a vertex $v \in \pi^{-1}(t')$ where $t'$ is the parent or a child of $t$;
    \item a rib to the top edge of $K^{r}$.
\end{itemize}

As observed above, edges of the former type lead to a vertex $v$ with $\tau(v)\neq a$. Therefore if $Q$ visits the bag of $t$ and other vertices of $Z$ then $Q$ follows a rib to the top edge of $K^{r}$.
As $Q$ may visit the top edge of $K^{r}$ at most twice we deduce that there are at most three nodes of $B_\ell(r)$ whose bag is intersected by $Q$ (those of $r$ and of two leaves of $B_\ell(r)$).
 
For each $j\in \{1,2\}$ the same argument applies to $s_j$ and the leaves of $B_\ell(s_j)$:  there are at most three nodes of $B_\ell(s_j)$ whose bags are intersected by $Q$, which are $s_j$ and at most $2$ leaves. We call these leaves $t,t'$ (and choose them arbitrarily if $Q$ visits less than two bags of leaves).

Each of $t$ and $t'$ has two children in $B_\ell$. For each such child $r$, we observed above (as $r\in D$) that $Q$ intersects at most 3 bags of nodes of $B_\ell(r)$.
This shows that among the bags of descendants of $s_j$, 
$Q$ intersects at most 15 of them. So in total $Q$ is contained in the union of at most 30 bags. The bound then follows from the fact that bags have bounded size (Remark~\ref{rem:preimage-partition}).
\end{proof}

\section{The induced paths of \texorpdfstring{$G_\ell$}{Gl} are short}
\label{sec:smallip}

An induced path $v_1\dots v_q$ of $G_\ell$ is called \emph{monotone} if $\depth(v_i)\leq \depth(v_j)$ for all $1\leq i<j\leq q$.
It is called \emph{bimonotone} if there exists $k\in \intv{1}{q}$ such that $v_k\dots v_1$ and $v_k\dots v_q$ are monotone.

\begin{lemma}\label{lem:bluegreen}
There is a constant $c_{\ref{lem:bluegreen}}$ such that the following holds.
Let $Q$ be an induced path in $G_\ell$ such that no source in $B_\ell$ is $Q$-special.
Then $|Q| \leq c_{\ref{lem:bluegreen}} \cdot \ell$.
\end{lemma}

\begin{proof}
Let $Q = v_1\dots v_q$ be such an induced path. 
The proof is split into three claims.

\begin{claim}\label{claim:bimonotone}
The path $Q$ is bimonotone
\end{claim}
\begin{proof}
By Lemma~\ref{lem:depth-root} there is a unique node $t\in V(B_\ell)$ minimizing the depth and whose bag is visited by $Q$.
Let $k$ be the minimum index such that $v_k$ belongs to the bag of $t$.
We now show that $\depth(v_k) \leq \dots \leq \depth(v_1)$, i.e., that $v_k\dots v_1$ is monotone. 
Let us assume toward a contradiction that it is not the case, and let $i\in \intv{1}{k}$ be the maximum integer such that $\depth(v_{i})>\depth(v_{i-1})$. 
Take $j\in \intv{i}{k}$ minimum such that $\depth(v_{j+1})<\depth(v_i)$.
By the choice of $i,j$ we note that the vertices in the subpath $v_i\dots v_j$ have equal depth. Because the edges in $G_\ell$ only connect vertices of bags corresponding to $\preceq$-comparable vertices of $B_\ell$, two adjacent vertices of equal depth necessarily belong to the same bag. Therefore, $\{v_i,\dots, v_j\}$ is a subset of the bag of a single node $s\in V(B_\ell)$.
Furthermore by Remark~\ref{rem:depth-decreasing} the edges $v_{i-1}v_i$ and $v_jv_{j+1}$ are either tree edges connecting a clique vertex to a top-predecessor, or ribs.

We distinguish cases depending on the type of the edges $v_{i-1}v_i$ and $v_jv_{j+1}$.

In the case where both $v_{i-1}v_i$ and $v_jv_{j+1}$ are tree edges, $v_{i-1}$ and $v_{j+1}$ belong to the same clique $K^{t'}$ (where $t'$ is the parent of $t$ in $B_\ell$), which contradicts the fact that $Q$ is induced.
We now distinguish cases depending on whether $v_{i-1}v_i$ or $v_jv_{j+1}$ is a rib.

In the case where $v_{i-1}v_i$ is a rib with $v_{i-1}$ in the top edge of some source $s\in V(B_\ell)$,
we consider the edge $v_jv_{j+1}$. Note that it cannot be a rib as otherwise $v_{i-1}$ and $v_{j+1}$ would both be part of the top edge of $K^s$, contradicting the fact that $Q$ is induced. So it is a tree edge. As $\depth(v_{j+1}) = \depth(v_i) - 1$, $\pi(v_{j+1})$ is an internal node of $B_\ell(s)$. Hence $s$ is $Q$-special, a contradiction.
The remaining case where $v_jv_{j+1}$ is a rib is symmetric.
This proves that such a $i$ does not exist. 
The proof of the monotonicity of $v_k\dots v_q$ is symmetric, yielding the bimonotonicity of $Q$.
\cqed
\end{proof}

By the virtue of Claim~\ref{claim:bimonotone} there is an integer $k \in \intv{2}{q}$ such that each of $v_k\dots v_1$ and $v_k\dots v_q$ is monotone. We choose such a $k$ minimum and set $k'$ as the maximum integer such that $v_k,\dots, v_{k'}$ all lie in a same bag and $v_{k'}\dots v_q$ is monotone, with possibly $k=k'$.

\begin{claim}\label{claim:tau-increasing}
For every $i,j \in \intv{k'+1}{q}$ with $i<j$, we have $\tau(v_i)\leq \tau(v_j)$.
\end{claim}

\begin{proof}
The proof is by contradiction.
Let $i\in \intv{k'+1}{q-1}$ be the minimum integer such that $\tau(v_{i+1})<\tau(v_i)$ and let $s = \pi(v_i)$.
By Lemma~\ref{lem:tau-decreasing} and since $\depth(v_i)\leq \depth(v_{i+1})$, $s$ is a source, $v_i$ belongs to the clique $K^s$, $v_iv_{i+1}$ is a tree edge and we get that $v_{i+1}$ belongs to the bag of a child of $s$. 
If $v_i$ is a vertex incident to the top edge of $K^s$ then $s$ is $Q$-special, a contradiction. 
Thus $v_i$ is the only vertex of $K^s$ that does not belong to its top edge. 
By monotonicity of $v_{k'}\dots v_q$ we have $\depth(v_{k'})\leq \depth(v_{i-1})\leq \depth(v_i)$ and so $v_{i-1}$ cannot be a top-predecessor of a child of $s$.
Hence $v_{i-1}$ must belong to the top edge of $K^s$, yielding that $s$ is $Q$-special, a contradiction.
\cqed
\end{proof}

A symmetric proof shows the following symmetric statement.
\begin{claim}\label{claim:tau-increasing2}
For every $i,j \in \intv{1}{k-1}$ with $i<j$, we have $\tau(v_i)\geq \tau(v_j)$.
\end{claim}

We are now ready to conclude the proof. Recall that the function $\tau$ has values in $\intv{1}{\ell+1}$. It is non-increasing on $v_1\dots v_{k-1}$ (Claim~\ref{claim:tau-increasing2}) and has values in $\{1,\dots, \ell\}$ (as $\tau(v_{k-1})<\tau(v_k)$, by minimality of $k$), non-decreasing on $v_{k'+1}\dots v_q$ (Claim~\ref{claim:tau-increasing}) and does not keep the same value on more than $c_{\ref{lem:tau-constant}}$ consecutive vertices (Lemma~\ref{lem:tau-constant}). 
Now since $v_k\dots v_{k'}$ lies in a single bag, it has order $\leq7$.
We conclude that $Q$ has order at most $c_{\ref{lem:tau-constant}}(2\ell+1)+7$, which is bounded from above by $c\cdot \ell$ for some constant~$c$, as $\ell\geq 1$. 
\end{proof}

We now show that if an induced path visits the top edge of the bag of a source~$s$, as well as a bag of an internal node of $B_\ell(s)$, then one of its sides lives in the vertices of the bags of $B_\ell(s)$. 

\begin{lemma}\label{eq:interval-subtree}
Let $Q=v_1\dots v_q$ be an induced path of $G_\ell$ with $v_1$ an endpoint of the top edge of $K^s$ for some source $s$ in $B_\ell$. If $s$ is a $Q$-special source, then $V(Q) \subseteq V( G_\ell(s))\setminus \pred(s)$.
\end{lemma}

\begin{proof}
Let $X_s$ be the set defined in Lemma~\ref{lem:separation}. Let $i$ be the minimum integer such that $\pi(v_i)$ is an internal node of $B_\ell(s)$. Such an integer exists as $s$ is $Q$-special.

We show that there is no subpath $Q'$ of $Q$ from $v_i$ to some vertex $v\notin V(G_\ell(s))\setminus\pred(s)$. Observe that this implies the claimed statement as it shows that $v_i\dots v_q$ does not leave $G_\ell - \pred(s)$, and also that the same is true of $v_i\dots v_1$. Towards a contradiction suppose that such a path exists. 
Then by Lemma~\ref{lem:separation}, $Q'$ intersects~$X_s$. 
Note that $Q'$ does not intersect $K^s$ if it is a subpath of $v_i\dots v_q$ as otherwise there would be an edge between $v_1$ and an internal vertex of $Q'$, whereas $Q$ is induced. 
In the case where $Q'$ is a subpath of $v_i\dots v_1$, it may intersect $K^s$ on at most two vertices, possibly $v_1$ and $v_2$ where it stops.
Thus $Q'$ cannot leave $G_\ell - \pred(s)$ through the clique~$K^s$.
So it intersects the bag of a leaf $t$ of $B_\ell(s)$ and, in order to reach $v_i$, follows one of $L(\pi(v_i), t)$ or $R(\pi(v_i), t)$ (as defined in Section~\ref{sec:blozup} when constructing the blow-up). Recall that $v_1$ has neighbors in both such paths (these neighbors are endpoints of ribs incident to $v_1$). This contradicts the fact that $Q$ is induced.
\end{proof}

\begin{lemma}\label{lem:twosou}
Let $Q$ be an induced path of $G_\ell$ and let $a \in \intv{1}{\ell}$. At most two sources of rank $a$ in $B_\ell$ are $Q$-special.
\end{lemma}

\begin{proof}
Let $Q=v_1\dots v_q$.
We may assume that $B_\ell$ has at least two $Q$-special sources of rank $a$ as otherwise the statement holds.
So there exist $1\leq i<j\leq \ell$ and two $Q$-special sources of rank $a$ that we call $s_i$ and $s_j$ and such that $v_i$ and $v_{j}$ are respectively incident to the top edges of $K^{s_i}$ and $K^{s_j}$. We choose them with $j-i$ minimum.
As $s_i$ is a $Q$-special source, by Lemma~\ref{eq:interval-subtree}, one of the following holds:
\[
\{v_1,\dots ,v_i\} \subseteq V(G_\ell(s_i)) \quad \text{or} \quad \{v_i, \dots, v_q\} \subseteq V(G_\ell(s_i)).
\]
The same argument applied to $s_{j}$ yields that one of the following holds:
\[
\{v_1,\dots ,v_{j}\} \subseteq V(G_\ell(s_{j})) \quad \text{or} \quad \{v_{j}, \dots, v_q\} \subseteq V(G_\ell(s_{j})).
\]

However, the trees $B_\ell(s_i)$ and $B_\ell(s_{j})$ only contain one source of rank $a$, respectively $s_i$ and~$s_{j}$. Therefore among the choices above only the two following outcome are possible:
\[
\{v_1,\dots ,v_i\} \subseteq V(G_\ell(s_i)) \quad \text{and}\quad \{v_{j}, \dots, v_q\} \subseteq V(G_\ell(s_{j})).
\]

So $v_1\dots, v_i$ and $v_{j} \dots v_q$ contribute to two $Q$-special sources of rank~$a$ in~$Q$. By minimality of $j-i$, the section of $Q$ between $v_i$ and $v_{j}$ does not contribute to any $Q$-special source of rank $a$. Hence at most two sources of rank $a$ of $B_\ell$ are $Q$-special, as claimed. 
\end{proof}

We are now ready to prove our main result, that we restate here.

\main*

\begin{proof}
We fix $\ell \in \N$ and consider the graph $G_\ell$. Let $q$ denote the maximum order of an induced path in $G_\ell$.
We first show the following.

\begin{claim}\label{cl:elldeux}
 There is a constant $c$ such that $q \leq c\ell^2$.  
\end{claim}
\begin{proof}
Let $Q$ be an induced path of $G_\ell$. By Lemma~\ref{lem:twosou}, for every $a \in \intv{1}{\ell}$ there are at most two sources of $B_\ell$ of rank $a$ that are $Q$-special. Therefore in total at most $2\ell$ sources of $B_\ell$ are $Q$-special.
Let $A \subseteq V(Q)$ be the set of vertices of $Q$ such that $\pi(v)$ is $Q$-special. For any node $s\in V(B_\ell)$, $|\pi^{-1}(s)| \leq 7$ (see Remark~\ref{rem:preimage-partition}) hence $|A| \leq 14\ell$.
We now consider a maximal subpath $Q'$ of $Q$ that does not intersect~$A$. 
By definition of $A$, no source of $B_\ell$ is $Q'$-special. Therefore by Lemma~\ref{lem:bluegreen}, $|Q'| \leq c_{\ref{lem:bluegreen}}\cdot \ell$. 
We conclude that $|Q| \leq |A| + (2\ell+1)\cdot c_{\ref{lem:bluegreen}}\cdot \ell \leq c\cdot \ell^2$ for some appropriately chosen constant $c$. Hence $q\leq c\ell^2$.\cqed
\end{proof}

Let $n = |G_\ell|$. As proved in Lemmas~\ref{lem:ham} and~\ref{lem:2deg}, $G_\ell$ has a path of order $n$ and has degeneracy $2$. Also, by Lemma~\ref{lem:taillede}, $n \geq 2^{2^{\ell+1}}$.
So $c(\log \log n)^2 \geq c(\ell+1)^2\geq q$ (by Claim~\ref{cl:elldeux}) proving the theorem.
\end{proof}

\section{Bounding the coloring numbers of \texorpdfstring{$G_\ell$}{Gl}}
\label{sec:colnum}
Given a total ordering $\sigma$ of the vertices of a graph $G$ and an integer $r\geq 1$, we say that a vertex $y$ is $r$-reachable from a vertex $x$ (with respect to $\sigma$) if $y<_\sigma x$ and there is a path $P$ of length (i.e., number of edges) at most $r$ such that every internal vertex $z$ of $P$ satisfies $x<z$.

The \emph{$r$-coloring number} $\col_r(G)$ is the minimum integer $k$ such that there is a total ordering $\sigma$ on $V(G)$ such that at most $k$ vertices are $r$-reachable from every vertex of~$G$~\cite{kierstead2003orderings}.

The value $\col_1(G)$ is exactly the degeneracy of $G$. We saw in Lemma~\ref{lem:2deg} that the $G_\ell$'s have degeneracy~2. Below we show that furthermore their coloring numbers are bounded from above by a linear function of~$r$, i.e., Theorem~\ref{th:linecol}, that we restate here for convenience.

\linecol*

\begin{proof}
Let $\ell\in \N_{\geq 1}$. We will prove that for every integer $r\geq 1$, $\col_r(G_\ell) \leq 2r+8.$
Let $\sigma$ be any total order on $V(G_\ell)$ such that for every $u,v \in V(G_\ell)$, if $\pi(u)\prec \pi(v)$ then $u<_\sigma v$. Let $r \in \N_{\geq 1}$.

    \begin{claim}\label{cl:sablier}
        Let $x_1 \in V(G_\ell)$ and let $x_1\dots x_p$ (for some $p\in \intv{2}{r+1}$) be a path witnessing that $x_p$ is $r$-reachable from $x_1$.
        Then $\pi(x_p) \preceq \pi(x_1)$ and for every $i\in \intv{2}{p-2}$,  $\pi(x_1) \preceq \pi(x_i)$.
    \end{claim}
    \begin{proof}
    We start with the first part of the statement. Towards a contradiction, let us suppose that $\pi(x_p)\npreceq \pi(x_1)$. By definition of reachability we have $x_p <_\sigma x_1$, so by definition of $\sigma$ we deduce that $\pi(x_1)$ and $\pi(x_p)$ are $\prec$-incomparable. 
    Recall that in $G_\ell$, if $uv$ is an edge then $\pi(u)$ and $\pi(v)$ are $\prec$-comparable. So in $G_\ell$, $x_1$ and $x_p$ are separated by the union of the bags of the common ancestors of $\pi(x_1)$ and $\pi(x_p)$. Hence for some $i\in\intv{2}{p-1}$, the node $\pi(x_i)$ is a common ancestor of $\pi(x_1)$ and $\pi(x_p)$. But then $x_i <_\sigma x_p$, a contradiction. Hence $\pi(x_p) \preceq \pi(x_1)$.
    The argument is similar for the second part of the statement: assuming it does not hold for some $i\in\intv{2}{p-1}$, we deduce that $\pi(x_1)$ and $\pi(x_i)$ are $\prec$-incomparable and thus there is a $j\in \intv{2}{i-1}$ such that $\pi(x_j)$ is a common ancestor of $\pi(x_1)$ and $\pi(x_i)$, implying the contradictory $x_i <_\sigma x_1$. \cqed
    \end{proof}

    Let $x$ in $V(G_\ell)$ and let $d$ be the depth of $x$.

    \begin{claim}\label{cl:8}
        At most 8 vertices $y$ are $r$-reachable from $x$ via a path whose edge incident to $y$ is not a rib.
    \end{claim}
    \begin{proof}
        Let $y$ be such a vertex and let $z$ be the other endpoint of the aforementioned edge.
        By definition we have $y <_\sigma x \leq_\sigma z$ so by Claim~\ref{cl:sablier} we get $\pi(y) \preceq \pi(x) \preceq \pi(z)$. Observe that if we ignore ribs, the graph in hand is merely the blow-up of a tree so the neighborhood of $z$ is included in the union of the bag of $\pi(z)$ and the bags of the nodes that are neighbors of $\pi(z)$. So the above inequality implies that $y$ is either in the bag of $\pi(x)$ or in the bag of the parent $t$ of $\pi(x)$. We can further reduce the choices for $y$ by observing that only two vertices of the bag of $t$ have neighbors in the bag of $\pi(x)$ and also by noting that $y\neq x$.
        This leaves at most $7+2-1=8$ possible choices for $y$, as claimed.
        \cqed
    \end{proof}

    Let $Y$ denote the set containing every vertex $y$ that is $r$-reachable from $x$ via a path whose edge incident to $y$ is a rib. We call $z_y$ the other endpoint of such an edge.
    So now in order to conclude the proof it is enough to prove the following statement.
    \begin{claim}\label{cl:2r}
        $|Y| \leq 2r$.
    \end{claim}
    \begin{proof}
    Recall that if $yz$ is a rib for some vertex $z$ then $(\depth(y), \depth(z)) \in \IS_\ell$.
    Let
    \[(i_1, i'_1), \dots, (i_k, i'_k) \in \IS_\ell\] denote the intervals of the form $(\depth(y), \depth(z_y))$ for $y\in Y$, ordered so that $i_1>\dots >i_k$.
    By the property that intervals of $\IS_\ell$ do not cross (Remark~\ref{rem:nestint}) and the fact that $\pi(x)\preceq \pi(z_y)$ for all such $z_y$ (Claim~\ref{cl:sablier}), we have the following ordering:
    \[
    i_k < \dots < i_1 \leq d \leq i'_1 <\dots i'_k.
    \]
    For the same reason there is no $(i,i') \in \IS_\ell$ with $d\leq i < i'_{j} < i'$ for some $j \in \intv{1}{k}$, so no rib of $G_\ell$ joins a vertex of depth $i$ to one of depth $i'$ for such $i,i'$. Therefore any path $P$ with vertices in $\{z\in V(G) : x \leq_\sigma z\}$ from $x$ to a vertex of depth $i'_k$ contains a vertex of depth $i'_j$ for every $j\in \intv{1}{k}$. This shows that $P$ has length at least $k-1$, hence $k\leq r$.

    Recall that by Claim~\ref{cl:sablier}, every $y\in Y$ belongs to the bag of an ancestor of $\pi(x)$. So $Y$ is included in the union of the bags of the ancestors of $\pi(x)$ of respective depths $i_1, \dots, i_k$.
    Furthermore, for each node $t\in V(B_\ell)$, only the two endpoints of the root edge of the bag of $t$ are incident to ribs to vertices in the bags of deeper nodes.
    So we get $|Y|\leq 2k \leq 2r$, as claimed.
    \cqed   
    \end{proof}
    From Claims~\ref{cl:8} and \ref{cl:2r} we get that $\col_{r}(G_\ell) \leq 2r+8$, as desired.
\end{proof}

\section{Discussion}
\label{sec:ohpeine}

In this paper we described a construction disproving Conjecture~\ref{conj:esp} by tightening the upper-bound on the function $f_k$ of Question~\ref{que:nodm} as follows:
\[
\forall n,k\in \N_{\geq 2},\quad \frac{\log \log n}{\log (k+1)} \quad \leq \quad  f_k(n) \quad \leq  \quad c (\log \log n)^2,
\]
for some constant $c>0$.
So far there is no evidence that the $O((\log \log n)^2)$ bound could be asymptotically tight so we did not try to optimize $c$ in our proofs.\footnote{It seems that the same construction starting with a $k$-ary tree instead of a binary one, we could win a $\frac{1}{\log(k+1)}$ factor for $k$-degenerate graphs.} 
We however note that our bound cannot be asymptotically improved by conducting a different analysis on the same construction, as our graphs contain induced paths of order $\Omega((\log \log n)^2)$; see Figure~\ref{fig:longlip} for an example.
The main open problem following this work is to determine the exact order of magnitude of~$f_k$. 
An other (possibly more difficult) problem is to find the optimal bound in Theorem~\ref{th:grs}.

Let us now ask a more general question.
For a graph class $\mathcal{G}$ that excludes some biclique as subgraph,\footnote{As discussed in the introduction, excluding arbirarlily large cliques and bicliques is necessary (in hereditary classes) for such a function to exist.} let $f_\mathcal{G}\colon \N\to\mathbb{R}$ be the maximum function such that for every $G\in \mathcal{G}$, if $G$ has a path of order $n$ then it has an induced path of order at least $f_\mathcal{G}(n)$.
In this paper we proved that there are classes $\mathcal{G}$ with linearly bounded coloring numbers (hence linearly bounded expansion) and where $f_\mathcal{G}(n) = O((\log \log n)^2)$. On the other hand, for several graph classes $\mathcal{G}$ it is known that $f_\mathcal{G}(n) = \Omega((\log n)^c)$ for some $c>0$. Such bounds exist for instance for planar graphs (Theorem~\ref{th:elm}), graph classes of bounded treewidth, more generally any class that excludes a topological minor (Theorem~\ref{th:claire}).
This exhibits two fundamentally different behaviors for the function~$f_\mathcal{G}$. In \cite{hilaire2022long} a third behavior is identified: graph classes where $f_\mathcal{G}$ is polynomial, such as graphs of bounded pathwidth.
It is therefore natural to ask: which criteria distinguish graph classes $\mathcal{G}$ such that $f_{\mathcal{G}}$ is polynomial, thereforepolylogarithmic, and doubly-polylogarithmic?

\begin{figure}[H]
    \centering
    \includegraphics[scale=1.1]{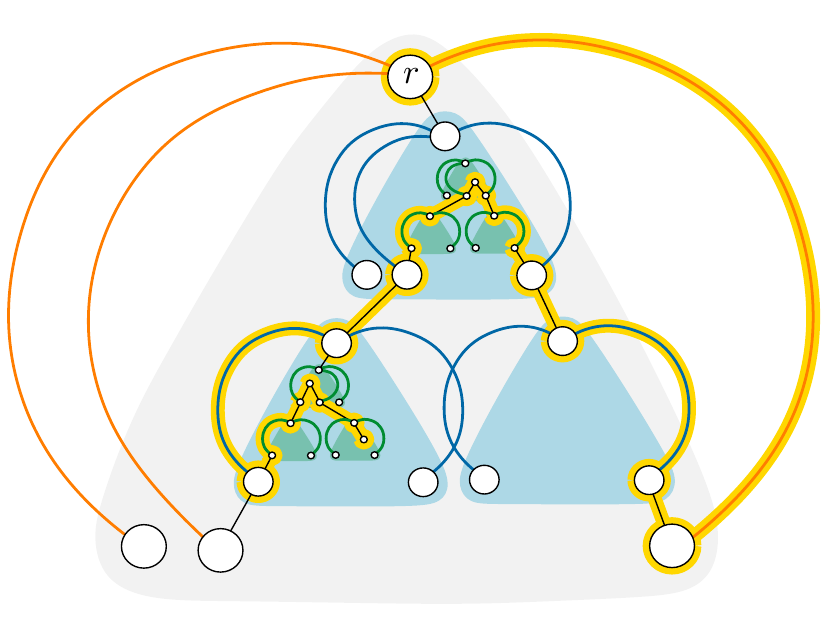}
    \caption{A representation (highlighted in gold) of an induced path of order $\Omega((\log \log n)^2)$  in $G_\ell$ starting from its root bag. For readability, only its "trace" in $B_\ell$ is represented together with the intervals (here drawn in a curved way) it follows, i.e., the actual path in $G_\ell$ may be obtained by replacing each highlighted node $s\in V(B_\ell)$ by appropriate vertices in $K^s$, each highlighted edge $ss'\in E(B_\ell)$ by an appropriate path $L(s,s')$ or $R(s,s')$, and each highlighted interval by an appropriate rib it originates from. A careful analysis shows that the order of such a path is $\Omega(\ell^2)$.}
    \label{fig:longlip}
\end{figure}


\section*{Acknowledgements}

We thank Louis Esperet for suggesting us to investigate whether our construction has bounded expansion, which led to Section~\ref{sec:colnum}.

\bibliographystyle{alpha}
\bibliography{main}

\end{document}